\documentclass[12pt]{article}

\usepackage[utf8]{inputenc}
\usepackage{todonotes}
\usepackage{amsmath}
\usepackage{amsthm}
\usepackage{amssymb}
\usepackage{algpseudocode}
\usepackage[boxruled,lined,linesnumbered,noend]{algorithm2e}
\usepackage{tikz}
\usetikzlibrary{automata}
\usetikzlibrary{arrows,shapes}

\usepackage{graphicx}
\usepackage{authblk}

\usepackage[USenglish]{babel}

\usepackage{amssymb,amsmath,verbatim,tabularx,enumitem}
\usepackage{colonequals,graphicx,psfrag}
\usepackage{amsmath,amsthm,amssymb}
\usepackage{subfigure}
\usepackage{natbib}
\usepackage{multirow}

\newcommand{\hc}{\hat{c}}
\newcommand{\N}{\mathbb{N}}

\newcommand{\I}{\mathcal{I}}

\newcommand{\U}{{\mathcal{U}}}
\newcommand{\Uc}{{\mathcal{U}}^\Gamma}

\newcommand{\R}{\mathbb{R}}
\newcommand{\X}{\mathcal{X}}
\newcommand{\XX}{\mathbf{X}}

\newcommand{\permx}{d^x}

\newcommand{\xx}{\mathbf{x}}

\newcommand{\x}[1]{x^{(#1)}}

\newcommand{\is}{\mathcal{O}}

\DeclareMathOperator*{\opt}{opt}

\newcolumntype{H}{>{\setbox0=\hbox\bgroup}c<{\egroup}@{}}

\newtheorem{theorem}{Theorem}
\newtheorem{lemma}[theorem]{Lemma}

\newtheorem{observation}[theorem]{Observation}

\begin{document}
\title{Faster Algorithms for Min-max-min Robustness for Combinatorial Problems with Budgeted Uncertainty}

\author[1]{Andr\'e Chassein}
\affil[1]{Data Analytics Center of Excellence, Deutsche Post DHL Group, Bonn, Germany}

\author[2]{Marc Goerigk}
\affil[2]{Network and Data Science Management, University of Siegen, Germany}

\author[3]{Jannis Kurtz}
\affil[3]{Department of Mathematics, RWTH Aachen University, Germany}

\author[4]{Michael Poss\thanks{Corresponding author. Email: michael.poss@lirmm.fr}}
\affil[4]{LIRMM, University of Montpellier, CNRS, France}

\date{}

\maketitle

\begin{abstract}
We consider robust combinatorial optimization problems where the decision maker can react to a scenario by choosing from a finite set of $k$ solutions. This approach is appropriate for decision problems under uncertainty where the implementation of decisions requires preparing the ground. We focus on the case that the set of possible scenarios is described through a budgeted uncertainty set and provide three algorithms for the problem. The first algorithm solves heuristically the dualized problem, a non-convex mixed-integer non-linear program (MINLP), via an alternating optimization approach. The second algorithm solves the MINLP exactly for $k=2$ through a dedicated spatial branch-and-bound algorithm. The third approach enumerates $k$-tuples, relying on strong bounds to avoid a complete enumeration. We test our methods on shortest path instances that were used in the previous literature and on randomly generated knapsack instances, and find that our methods considerably outperform previous approaches. Many instances that were previously not solved within hours can now be solved within few minutes, often even faster.
\end{abstract}


\baselineskip 20pt plus .3pt minus .1pt

\section{Introduction}

Real-world problems are uncertain, and optimization approaches need tools to reflect this uncertainty. One such approach is robust optimization, which dates back to the seminal work of~\cite{soyster1973}. Since the breakthroughs arising about twenty years ago \citep{roconv,bentalRobust,kouvelis2013robust,el1998robust}, robust optimization has become a key framework to address the uncertainty that arises in optimization problems. The rationale behind robust optimization is to characterize the uncertainty over unknown parameters through a set which contains all relevant scenarios and to measure the worst-case over this set. One of the main reasons for the success of robust optimization is its tractability. For instance, linear robust optimization problems are essentially as easy as their deterministic counterparts for many types of convex uncertainty sets, see \cite{roconv}, contrasting with the well-known difficulty of stochastic optimization approaches. For general overviews on the field, we refer to \cite{aissi_minmax_survey,bertsimas2011theory,goerigk2016algorithm,gabrel2014recent}.

The picture is more complex when it comes to robust combinatorial optimization problems. Let $[n]=\{1,\ldots,n\}$ denote a set of indices and $\X\subseteq \{0,1\}^{n}$ be the feasibility set of a combinatorial optimization problem. Given a bounded uncertainty set $U\subseteq \R^n$, the classical robust counterpart of the problem $\min_{x\in\X} \sum_{i\in[n]} c_i x_i$ is
\begin{equation}
\label{eq:minmax}\tag{M$^2$}
 \min_{x\in\X} \max_{c\in \U}\sum_{i\in[n]} c_i x_i.
\end{equation}
It is well known (e.g. \cite{aissi_minmax_survey,kouvelis2013robust}) that a general uncertainty set $\U$ leads to a robust problem that is, more often than not, harder than the deterministic problem. Robust combinatorial optimization witnessed a breakthrough with the introduction of budgeted uncertainty by~\cite{BertsimasS03} (also known as $\Gamma$-uncertainty), which keeps the tractability of the deterministic counterpart for a large class of combinatorial optimization problems. Specifically, \cite{BertsimasS03} considered uncertain cost functions characterized by the vector $\hc\in\R^n$ of nominal costs and the vector $d\in\R^n_+$ of deviations. Then, given a budget of uncertainty $\Gamma>0$, they addressed uncertainty sets of the form
\[
\Uc = \left\{ c\in\R^n : c_i = \hc_i + d_i z_i,\ z\in\left[0,1\right]^n,\ \sum\limits_{i\in[n]} z_i \le \Gamma \right\}.
\]
\cite{BertsimasS03} showed how the optimal solution of problem \eqref{eq:minmax} can be obtained by solving $n+1$ deterministic counterparts of the problem. Several subsequent papers have reduced this number of deterministic problems \citep{Alvarez-MirandaLT13,LeeLPP12}, down to $\lceil\frac{n-\Gamma}{2}\rceil+1$ in \cite{Lee2014}, and extended the result to more general uncertainty polytopes \citep{Poss17}. We refer to \cite{buchheimrobust,kasperski2016robust} for recent surveys on robust combinatorial optimization.

Our focus in this paper is the alternative robust model introduced by \cite{Buchheim2017,JannisDiscrete} which is based on the idea of $k$-adaptability first introduced by \cite{bertsimas2010finite} for general robust two-stage problems. Later this idea was studied for robust two-stage problems with binary recourse in \cite{hanasusanto2015k} and for the same problems with mixed-integer recourse in \cite{subramanyam2017k}. In the latter publications both cases, uncertainty only affecting the objective function and uncertainty affecting the constraint coefficients, are studied. In contrast to this, the approach of \cite{Buchheim2017,JannisDiscrete} is limited to the case of objective uncertainty and binary decision variables. Furthermore the authors do not consider robust two-stage problems, containing first-stage and second-stage decisions, but apply the idea of $k$-adaptability to classical combinatorial problems. The main idea of the approach is that the decision maker \emph{prepares} $k$ solutions from $\X$ before knowing the scenario $c$. Then, upon full knowledge of $c$, the decision maker can choose the cheapest of the $k$ solutions that had been prepared. Using the robust paradigm the aim is to find $k$ solutions which perform well in the worst case over all scenarios while in the objective function for each scenario the best of the $k$ solutions is considered. This idea results in the problem
\begin{equation}\label{eq:minmaxmin_dU}\tag{M$^3$}
\min_{\x{1}, \ldots, \x{k} \in \X} \max_{c\in \U} \min_{j\in[k]} \sum_{i\in[n]} c_i \x{j}_i.
\end{equation}
The approach modeled by \eqref{eq:minmaxmin_dU} is typically useful in applications where some groundwork has to be made ahead of knowing the data. An example taken from \cite{hanasusanto2015k} is related to disaster management wherein one must be able to transport relief supplies or evacuating citizens in case an uncertain disaster arises (see e.g. \cite{chang2007scenario,liberatore2013uncertainty}). Here the storage locations of the supplies and possible evacuation paths have to be determined in advance. The optimized set of emergency plans, $\x{1},\ldots,\x{k}\in \X$ for a small number $k$, ought to be planned and trained for well before the disaster happens.

Similar applications arise in the context of transportation problems in logistics (e.g. the Hub-Location Problem; see \cite{alumur2012hub}). Here a company may have to make decisions in advance (e.g. reserving a fleet of trucks or parts of a railing system owned by the government, constructing facilities or hubs etc.) to provide a working supply-chain in the future. Since the future demands of the customers are uncertain, a flexible transportation system is very useful. Using the min-max-min approach a small number of transportation-plans can be calculated in advance to decide which long-term decision have to be made and to prepare all employees. Another example, taken from \cite{eufinger2017robust,subramanyam2017k}, considers a parcel service delivering to the same customers every day, i.e. $\X$ is the set of feasible solutions of the vehicle routing problem. At the beginning of each day, the company determines a route taking into account the current traffic situation. In this case again, the drivers need time to be trained for the set of possible routes, to avoid, for instance being stuck in narrow streets with large vehicles. Hence, the set of candidate routes should be small and known ahead of the departures of the drivers.

When the number of solutions is large ($k\geq n+1$), \cite{Buchheim2017} essentially show for general convex uncertainty sets that \eqref{eq:minmaxmin_dU} is not harder than its deterministic counterpart. Unfortunately, in applications that require preparing the ground, it is not practical to have too many alternatives, limiting the interest of the approach from \cite{Buchheim2017}, which requires $k\geq n+1$. Alternative solutions to \eqref{eq:minmaxmin_dU} have also been proposed in a much more general context where there are also decisions that must be taken before the uncertainty is revealed, falling into the framework of two-stage robust optimization, see~\cite{hanasusanto2015k} and~\cite{subramanyam2017k}. The former provides a MILP reformulation involving the linearization of products between binary and real variables, while the latter studies an ad-hoc branch-and-bound algorithm that branches over the assignments of solutions to scenarios. Unfortunately, these two approaches are able to prove optimality only for the smaller instances studied therein. Considering the case of budgeted uncertainty and $k=2$, \cite{Chassein2017} proved that Problem \eqref{eq:minmaxmin_dU} can be solved in polynomial time for the matroid maximization problem, the selection problem and the unconstrained problem while it is strongly NP-hard for the shortest path problem. Despite the theoretical efficiency for several combinatorial problems the procedure derived in \cite{Chassein2017} is not efficient for practical purposes. Furthermore general algorithms applicable even for the NP-hard cases are desired.

The purpose of this work is to overcome these limitations in the case of small $k$ (typically 2 or 3), proposing efficient exact  (and heuristic) algorithms for the resulting optimization problems. Our algorithms are tailored for the budgeted uncertainty set $\Uc$ for two main reasons. First, budgeted uncertainty has been successfully used in numerous applications, including transport and logistics \citep{AgraCHR18,lee2012robust}, energy production \citep{bertsimas2013adaptive}, telecommunications network design \citep{KosterKR13,LeeLPP12}, portfolio selection \citep{KawasT17}, among many others. Second, the specific structure of the set can be leveraged to provide efficient algorithms.
The contributions of this paper can be summarized as follows:
\begin{itemize}
\item We propose a local-search heuristic based on the dualized non-linear reformulation, valid for any value of $k$.
\item We solve the non-linear reformulation exactly through a spatial branch-and-bound algorithm, valid for the case $k = 2$. Our algorithm relies on strong lower bounds, tailored for $\Uc$.
\item We provide an enumeration algorithm to solve the problem for small values of $k$, typically 2 or 3. Leveraging the structure of $\Uc$, as well as ad-hoc upper and lower bounds, the algorithm is able to enumerate a small subset of the $k$-tuples to prove optimality.
\item Using shortest path instances from the literature and new randomly generated knapsack instances, we show that our methods are able to improve computation times considerably, solving problems to optimality within minutes (often seconds) that were previously unsolved in hours. For $k=4$, our heuristic provides solutions close to those obtained in the literature in small amounts of time.
\end{itemize}

The remainder of this paper is structured as follows. The algorithms based on the non-linear reformulation are presented in Section~\ref{sec:algo:convex} while the general enumerative algorithm is described in Section~\ref{sec:algo:general}. Computational experiments are discussed in Section~\ref{sec:experiments}, before concluding the paper in Section~\ref{sec:conclusions}.

\paragraph{Notations.} For any integer $n$, we denote the set $\{1,\ldots,n\}$ as $[n]$. Further, the $k$-tuple $(\x{1},\ldots,\x{k})$ is shortened to $\xx$, and $\X^k$ denotes the Cartesian product $\times_{i=1}^k \X$.

\section{Non-Linear Algorithms}
\label{sec:algo:convex}

\subsection{Problem Reformulation}
\label{sec:problemref}

Our first two algorithms address problem \eqref{eq:minmaxmin_dU} as a non-convex MINLP. Introducing the optimization variable $z$ to express the inner minimization problem of \eqref{eq:minmaxmin_dU} leads to the following min-max problem
\begin{equation}\label{eq:minmaxproblem}
\begin{aligned}
\min_{\xx\in\X^k} \max_{c\in \Uc} \max_{z} &\left\{ z : z \leq \sum_{i\in[n]} c_i \x{j}_i, \forall j\in [k]\right\}\\
= \min_{\xx\in\X^k} \max_{c,z} & \left\{ z : z \leq \sum_{i\in[n]} c_i \x{j}_i, \forall j\in [k], c\in \Uc\right\}.
 \end{aligned}
 \end{equation}
Dualizing the inner maximization problem, which is a linear optimization problem, we obtain the following non-linear compact formulation for~\eqref{eq:minmaxmin_dU}:
\begin{equation}\label{eq:nonlinearformulation}\tag{NL}
\begin{aligned}
\min_{\xx,\theta,\gamma,\alpha}\ & \sum_{j\in[k]} \sum_{i\in[n]} \hc_i \x{j}_i \alpha_j + \Gamma \theta + \sum_{i\in[n]} \gamma_i \\
\text{s.t. } & \theta + \gamma_i \ge \sum_{j\in[k]} d_i \x{j}_i \alpha_j & \forall i\in[n] \\
& \sum_{j\in[k]} \alpha_j = 1 \\
& \theta \ge 0 \\
& \gamma_i \ge 0 & \forall i\in[n] \\
& \alpha_j \ge 0 & \forall j \in[k] \\
& \xx\in\X^k .
\end{aligned}
\end{equation}
Notice that the above formulation can be linearized by replacing the product $\alpha_j\x{j}_i$ by a new variable $z_i^j$, which is restricted by the constraints $z_i^j\geq 0$ and $z_i^j \ge \alpha_j + \x{j}_i -1$. This leads to a compact mixed-integer programming formulation (MIP) \citep{hanasusanto2015k}.

\cite{Chassein2017} showed that it suffices to enumerate a finite set of $O(n^{2k-1})$ many values for $\alpha$ and $\theta$, to solve Problem \eqref{eq:nonlinearformulation} exactly. Note that if $\alpha$ and $\theta$ is fixed the non-linearity vanishes and the problem reduces to an MIP with a certain structure. The following theorem, proved in \cite{Chassein2017}, summarizes this result in detail.
\begin{theorem}{\cite{Chassein2017}}\label{main_thm}
Problem~\eqref{eq:minmaxmin_dU} with budgeted uncertainty can be solved by solving at most $O(n^{2k-1})$ subproblems of the form
\[\min_{\xx\in\X^k} \sum_{i\in[n]} f_i(\x{1}_i,\dots,\x{k}_i) \tag{$P_{sub}$} \]
where
\[f_i(\x{1}_i,\dots,\x{k}_i) := \sum_{j\in[k]}  \alpha_j \hc_i\x{j}_i + \max\left(0,\sum_{j\in[k]} \alpha_j d_i\x{j}_i  - \theta\right)\]
for some fixed values $\alpha\in\R^k_+$ and $\theta\in\R_+$.
\end{theorem} 
\cite{Chassein2017} proves that Problem ($P_{sub}$) can be solved in polynomial time for the matroid maximization problem, the selection problem, the unconstrained problem and the shortest path problem on series-parallel graphs. Despite these positive results the author shows that Problem \eqref{eq:minmaxmin_dU} is strongly NP-hard for the shortest path problem in general.

\subsection{Local Search Heuristic}\label{sec:local}

The nonlinear part of model \eqref{eq:nonlinearformulation} is due to the product between $\x{j}$ and $\alpha$. A simple idea to avoid the nonlinearity is to search for local instead of global minima by considering only restricted search directions. Methods of this type are also known as block-coordinate descent algorithms, see, e.g. \cite{wright2015coordinate}. Instead of minimizing $\xx,\alpha,\gamma,$ and $\theta$ simultaneously, either we solve \eqref{eq:nonlinearformulation} only over the variables $\xx,\gamma,\theta$ for fixed values of $\alpha$, or we solve \eqref{eq:nonlinearformulation} over the variables $\alpha,\gamma,\theta$ and keep $\xx$ fix. The first optimization problem is called $x$-step, the second $\alpha$-step.  
To solve an $x$-step, we solve the following MIP
\begin{align*}
\min_{\xx,\gamma,\theta} \ & \sum_{j\in[k]} \sum_{i\in[n]} \alpha^j \hc_i \x{j}_i  + \sum_{i\in[n]} \gamma_i + \Gamma \theta \tag{$x$-step} \\ 
\text{s.t. } & \sum_{j\in[k]} \alpha^j d_i\x{j}_i  - \theta \leq \gamma_i &\forall i \in [n] \\
& \gamma_i  \ge 0 &\forall i \in [n] \\
& \theta \ge 0 \\
& \xx \in \X^k & \forall j\in[k]
\end{align*}
To solve an $\alpha$-step, we solve the following LP
\begin{align*}
\min_{\alpha,\gamma,\theta} \ & \sum_{j\in[k]} \sum_{i\in[n]} \alpha^j \hc_i \x{j}_i  + \sum_{i\in[n]} \gamma_i + \Gamma \theta \tag{$\alpha$-step} \\ 
\text{s.t. } & \sum_{j\in[k]} \alpha_j = 1 \\
& \sum_{j\in[k]} \alpha^j d_i \x{j}_i  - \theta \leq \gamma_i &\forall i \in [n] \\
& \gamma_i \ge 0 &\forall i \in [n] \\
& \alpha_j \ge 0 &\forall j \in [k] \\
& \theta \ge 0
\end{align*}
We start the local search with an $x$-step. As initial values for $\alpha$ we choose $\tilde{\alpha}^j = \frac{2j}{k(k+1)}$ for all $j \in [k]$. Note that $\sum_{j\in[k]} \tilde{\alpha}^j=1$. Different values for $\tilde{\alpha}$ help to break the symmetry of the model formulation. The optimal solution of the $x$-step is then used to solve the first $\alpha$-step. We iterate between $x$- and $\alpha$-steps until no further improvement is found. Note that we can use the optimal solution of an $x$-step to warm start the next $x$-step. Since the objective value decreases in each step, except of the last step, we will end up in a local minimum after a finite number of steps. 

\subsection{A Branch-and-Bound Algorithm for $k=2$}\label{sec:bb}

If $k=2$, Problem~(NL) can be rewritten in the following way (see also \cite{Chassein2017}):
\begin{equation}\label{eq:NL-2}\tag{NL-2}
\begin{aligned}
\min_{x,y,\alpha,\gamma,\theta} \ & \sum_{i\in[n]} \hc_i x_i \alpha + \sum_{i\in[n]} \hc_i y_i (1-\alpha) + \sum_{i\in[n]} \gamma_i + \Gamma \theta  \\ 
\text{s.t. } & d_ix_i \alpha + d_iy_i (1-\alpha)- \theta \leq  \gamma_i &\forall i \in [n] \\
& \gamma_i \ge 0 &\forall i \in [n] \\
& \alpha\in[0,0.5] \\
& \theta \ge 0\\
& x,y \in \X
\end{aligned}
\end{equation}
\noindent
In the following we define the optimal value of Problem \eqref{eq:NL-2} for a fixed value of $\alpha\in [0, 0.5]$ by $h(\alpha)$. Hence, our original problem can be solved if we can solve the problem $\min_{\alpha \in [0,0.5]} h(\alpha)$. From Theorem~\ref{main_thm}, we know that the candidate set $\mathcal{A}$ of optimal values for $\alpha$ is a finite set with size $O(n^3)$. Hence, a possible algorithm for the problem is to evaluate $h(\alpha)$ for each $\alpha \in \mathcal{A}$ and choose the best solution. However, solving $O(n^3)$ of these MIPs can be too time consuming. Using the structure of $h$ we can find the global minimum without evaluating $h(\alpha)$ for each $\alpha \in \mathcal{A}$.

The idea of the algorithm is to use a branch-and-bound strategy on the $\alpha$ variable and to divide the interval $[0,0.5]$ into smaller intervals. For each unexplored interval we calculate lower bounds which are described in detail below. If the list of unexplored intervals is empty, the algorithm has found the optimal solution. There are two reasons which allow to discard an interval. First, if for  an interval $\I$ it holds $\mathcal{A} \cap \I= \emptyset$ then it can be discarded since we know that the optimal solution is attained for an $\alpha \in \mathcal{A}$. Second if the lower bound for the actual interval exceeds the currently best solution, the interval can be discarded as well. 

If we cannot discard an interval $\I=[\alpha_1,\alpha_2]$ we evaluate $h(\tilde{\alpha})$ for some $\tilde{\alpha} \in \I$. This allows us to split $\I$ into two smaller sub intervals $[\alpha_1,\tilde{\alpha}]$ and $[\tilde{\alpha},\alpha_2]$. These intervals are then added to the list of unexplored intervals. It is possible that $h(\tilde{\alpha})$ improves the current best solution, which leads to an improved upper bound. 

To get a good feasible solution at the start of the algorithm, we use the local search heuristic from Section~\ref{sec:local} to find a local minimum $h(\alpha^*)$ at $\alpha^*$. The first two intervals of the list of unexplored intervals are then given as $[0,\alpha^*]$ and $[\alpha^*,0.5]$.

For the effectiveness of this algorithm the computation of the lower bound is crucial. We argue in the following how to derive a strong lower bound which is still reasonable to compute.

It was shown in \cite{Chassein2017} that $h(\alpha) = \min_{x,y\in \X} g(x,y,\alpha)$ where
\[g(x,y,\alpha) = \alpha \sum_{i\in[n]} \hc_ix_i + (1-\alpha) \sum_{i\in[n]}\hc_i y_i +  \begin{Vmatrix}\begin{pmatrix} d_1x_1 \alpha + d_1y_1(1-\alpha) \\ \vdots \\ d_nx_n \alpha + d_ny_n(1-\alpha)  \end{pmatrix}\end{Vmatrix}^{(\Gamma)}\]
and $||v||^{(\Gamma)}$  is the sum of the $\Gamma$ largest values of vector $v$. Note that $g(x,y,\alpha)$ is a piecewise affine-linear function in $\alpha$ where the breakpoints are the values of $\alpha$ for which the order of the components in $||\cdot||^{(\Gamma)}$ changes. If we increase $\alpha$, clearly each time when the order changes the slope of the next segment must increase. Therefore $g$
is convex in $\alpha$. Hence, we have that 
\[g(x,y,\alpha) \geq g(x,y,\alpha_0) + (\alpha-\alpha_0) \partial g(x,y,\alpha_0)\]
where $\partial g(x,y,\alpha_0)$ is a subdifferential for $g$. Recall that $\partial g$ is given by 
\[\partial g(x,y,\alpha_0) = \sum_{i\in[n]} \hc_i x - \sum_{i\in[n]} \hc_iy_i + \sum_{i \in I : x_i=1, y_i=0} d_i -  \sum_{i \in I: x_i=0, y_i=1} d_i\]
where $I$ is the set of the $\Gamma$ largest indices of $\begin{pmatrix} d_1x_1 \alpha_0 + d_1y_1(1-\alpha_0) \\ \vdots \\ d_nx_n \alpha_0 + d_ny_n(1-\alpha_0)  \end{pmatrix}$. The following two estimations are essential to compute the lower bound:
\begin{equation*}
\partial g(x,y,\alpha_0) \geq  \sum_{i\in[n]} \hc_ix_i - \sum_{i\in[n]} \hc_iy_i - \begin{Vmatrix}\begin{pmatrix} d_1y_1 \\ \vdots \\ d_ny_n \end{pmatrix}\end{Vmatrix}^{(\Gamma)} =: \underline{\partial g}(x,y) 
\end{equation*}
and
\begin{equation*}
\partial g(x,y,\alpha_0) \leq  \sum_{i\in[n]} \hc_ix_i - \sum_{i\in[n]} \hc_iy_i + \begin{Vmatrix}\begin{pmatrix} d_1x_1 \\ \vdots \\ d_nx_n \end{pmatrix}\end{Vmatrix}^{(\Gamma)} =: \overline{\partial g}(x,y) .
\end{equation*}
Given an interval $\I = [\alpha_1,\alpha_2]$ for which we want to find a lower bound value $L(\I)$ with $L(\I) \leq \min_{\alpha \in [\alpha_1,\alpha_2]} h(\alpha)$ the idea is to solve the following two minimization problems
\begin{equation} \label{eq:prob1proof} \min_{x,y} g(x,y,\alpha_1) + (\alpha_2-\alpha_1) \underline{\partial g}(x,y) \end{equation}
and
\begin{equation} \label{eq:prob2proof} \min_{x,y} g(x,y,\alpha_2) + (\alpha_1-\alpha_2) \overline{\partial g}(x,y). \end{equation}

\begin{lemma}
Let $(x^*_1,y^*_1)$ be an optimal solution of problem~\eqref{eq:prob1proof} and $(x_2^*,y_2^*)$ an optimal solution of problem~\eqref{eq:prob2proof}. For all $\alpha \in [\alpha_1,\alpha_2]$ it holds that 
\[ h(\alpha) \geq h(\alpha_1) + (\alpha-\alpha_1)\underline{\partial g}(x_1^*,y_1^*)\]  and
\[ h(\alpha) \geq h(\alpha_2) + (\alpha-\alpha_2)\overline{\partial g}(x_2^*,y_2^*).\]
\label{lem_new_lbs}
\end{lemma}
\begin{proof}
Let $\alpha \in [\alpha_1,\alpha_2]$ be fix. Let $(x^*,y^*)$ be a solution which defines $h(\alpha)$. For the sake of contradiction, assume that 
$h(\alpha) = g(x^*,y^*,\alpha) <  h(\alpha_1) + (\alpha-\alpha_1)\underline{\partial g}(x_1^*,y_1^*)$. First, observe that
\begin{align*}
h(\alpha_1) +(\alpha-\alpha_1)\underline{\partial g}(x_1^*,y_1^*) 
					&> g(x^*,y^*,\alpha) 	 \\
					&\geq g(x^*,y^*,\alpha_1) + (\alpha - \alpha_1) \partial g(x^*,y^*) \\
					&\geq g(x^*,y^*,\alpha_1) + (\alpha - \alpha_1) \underline{\partial g}(x^*,y^*) \\
					&\geq h(\alpha_1) + (\alpha - \alpha_1) \underline{\partial g}(x^*,y^*) 
\end{align*}
From which follows that $\underline{\partial g}(x_1^*,y_1^*) > \underline{\partial g}(x^*,y^*)$. Further, we have that
\begin{align*}
g(x^*_1,y^*_1,\alpha_1) +(\alpha-\alpha_1)\underline{\partial g}(x_1^*,y_1^*) &\geq h(\alpha_1) +(\alpha-\alpha_1)\underline{\partial g}(x_1^*,y_1^*)\\
					&> g(x^*,y^*,\alpha) 	 \\
					&\geq g(x^*,y^*,\alpha_1) + (\alpha - \alpha_1) \partial g(x^*,y^*) \\
					&\geq g(x^*,y^*,\alpha_1) + (\alpha - \alpha_1) \underline{\partial g}(x^*,y^*).
\end{align*}
Since $\underline{\partial g}(x_1^*,y_1^*) > \underline{\partial g}(x^*,y^*)$, we can add on the left hand side of this inequality chain $(\alpha_2-\alpha)\underline{\partial g}(x_1^*,y_1^*)$ and on the right hand side $(\alpha_2-\alpha)\underline{\partial g}(x^*,y^*)$ and obtain
\begin{align*}
g(x^*_1,y^*_1,\alpha_1) +(\alpha_2-\alpha_1)\underline{\partial g}(x_1^*,y_1^*) > g(x^*,y^*,\alpha_1) + (\alpha_2 - \alpha_1) \underline{\partial g}(x^*,y^*)	
\end{align*}
This gives the desired contradiction since $(x^*_1,y^*_1)$ is an optimal solution for problem~\eqref{eq:prob1proof}.

The second inequality can be proved analogously.
\end{proof}

\noindent
Using Lemma~\ref{lem_new_lbs} we obtain the following result.

\begin{theorem}
Let $(x^*_1,y^*_1)$ be an optimal solution of problem~\eqref{eq:prob1proof} and $(x_2^*,y_2^*)$ an optimal solution of problem~\eqref{eq:prob2proof}. For $\I=[\alpha_1,\alpha_2]$ a lower bound $L(\I)$ is given by
\[L(\I) = h(\alpha_1) + \left( \frac{h(\alpha_2)-h(\alpha_1) + \alpha_1\underline{\partial g}(x_1^*,y_1^*) + \alpha_2\overline{\partial g}(x_2^*,y_2^*)}{\underline{\partial g}(x_1^*,y_1^*) +\overline{\partial g}(x_2^*,y_2^*)} -\alpha_1 \right)\underline{\partial g}(x_1^*,y_1^*). \]
\label{cor_L}
\end{theorem}
\begin{proof}
Let $\alpha \in [\alpha_1,\alpha_2]$. We know from Lemma~\ref{lem_new_lbs} that $h(\alpha) \geq f_1(\alpha)$ and $h(\alpha) \geq f_2(\alpha)$ where $f_1$ and $f_2$ are the two linear functions given in Lemma \ref{lem_new_lbs} . We conclude, that $h(\alpha) \geq \max\{f_1(\alpha),f_2(\alpha)\}$. Hence, $\min_{\alpha \in [\alpha_1,\alpha_2]} h(\alpha) \geq \min_{\alpha \in [\alpha_1,\alpha_2]} \max\{f_1(\alpha),f_2(\alpha)\}$. Note that the value of the right hand side is given by $f_1(\alpha')$ where $f_1(\alpha') = f_2(\alpha')$. Using the formulas for $f_1$ and $f_2$, we obtain that $L(\I) = f_1(\alpha')$.
\end{proof}

In our branch-and-bound procedure we will use value $\alpha'$ from the proof above as a candidate to split the interval $[\alpha_1,\alpha_2]$ into two smaller intervals $[\alpha_1,\alpha']$ and $[\alpha',\alpha_2]$. 

To use Theorem~\ref{cor_L} in the branch-and-bound algorithm, we need to know the following four values: $h(\alpha_1)$, $h(\alpha_2)$, $\underline{\partial g}(x_1^*,y_1^*)$, and $\overline{\partial g}(x_2^*,y_2^*)$. The first two values are already computed by the algorithm, since we compute $h(\alpha')$ if we split an interval $[\alpha_1,\alpha_2]$ into two smaller intervals $[\alpha_1,\alpha']$ and $[\alpha',\alpha_2]$. Hence we know for each unexplored interval the value of $h$ at the boundaries of this interval (at the start of the algorithm we also compute $h(0)$ and $h(0.5)$). To compute $\underline{\partial g}(x_1^*,y_1^*)$, and $\overline{\partial g}(x_2^*,y_2^*)$, we need to solve problems~\eqref{eq:prob1proof} and \eqref{eq:prob2proof}. Each of these problems can be formulated as an MIP, which is explained in the following.

For fixed $x,y$ the value of $g(x,y,\alpha_1)$
can be represented by Problem \eqref{eq:NL-2} fixing $\alpha=\alpha_1$.
Next, consider the value of 
\[\underline{\partial g}(x,y) = \hat c^\top x - \hat c^\top y - \begin{Vmatrix}\begin{pmatrix} d_1y_1 \\ \vdots \\ d_ny_n \end{pmatrix}\end{Vmatrix}^{(\Gamma)}.\]
We introduce variables $\beta_i \in [0,1] $ for $i \in [n]$ which indicate the $\Gamma$ largest entries of $(d_1y_1,\ldots, d_ny_n)$. Therefore calculating $\underline{\partial g}(x,y)$ results in the following minimization problem
\begin{equation}\label{eq:lowbardeltag}
\begin{aligned}
\min_{\beta} \ &\hat c^{\top}x - \hat c^{\top}y - \sum_{i\in[n]} d_i\beta_i \\ 
\text{s.t. } & \sum_{i\in[n]} \beta_i \leq \Gamma \\
& 0 \leq \beta_i \leq y_i \quad \forall i \in [n] .
\end{aligned}
\end{equation}
Recall that the objective function of problem~\eqref{eq:prob1proof} is given by $g(x,y,\alpha_1) + (\alpha_2-\alpha_1) \underline{\partial g}(x,y)$.
Therefore substituting Formulations \eqref{eq:NL-2} for fixed $\alpha=\alpha_1$ and Formulation \eqref{eq:lowbardeltag} in Problem \eqref{eq:prob1proof} we obtain the equivalent MIP formulation
\begin{align*}
\min_{x,y,\beta,\gamma,\theta} \ &\sum_{i\in[n]}\hc_ix_i \alpha_2 + \sum_{i\in[n]} \hc_iy_i (1-\alpha_2) + \sum_{i\in[n]} \gamma_i + \Gamma \theta + \sum_{i\in[n]} (\alpha_1-\alpha_2)d_i\beta_i \\ 
\text{s.t. } & d_ix_i \alpha_1 + d_iy_i (1-\alpha_1)- \theta \leq  \gamma_i &\forall i \in [n] \\
& 0 \leq \gamma_i  &\forall i \in [n] \\
& 0 \leq \theta \\
& \sum_{i\in[n]} \beta_i \leq \Gamma \\
& 0 \leq \beta_i \leq y_i &\forall i \in [n]  \\
& x,y \in \X .
\end{align*}
Analogously Problem~\eqref{eq:prob2proof} can be reformulated as an MIP. This concludes the discussion on how to compute a lower bound $L(\I)$. We summarize the described procedure in Algorithm~\ref{alg:exact_2_adapt}. Note that it is possible to adapt the na\"ive implementation of this algorithm to make it computationally more effective in practice. For example, whenever the current best solution is improved by evaluating $h(\alpha')$, we can restart the local search heuristic at $\alpha'$ to find a new local minimum.

\begin{algorithm}
\SetKwInOut{Return}{Return}
\caption{Branch and bound algorithm (BB) with $k=2$.}
\label{alg:exact_2_adapt}
Compute the candidate set $\mathcal{A}$\;
Use the local search heuristic to find a local minimum $h(\alpha^*)$ at $\alpha^*$\;
Initialize the list of unexplored intervals $\mathcal{L} = \{ [0,\alpha^*],[\alpha^*,0.5] \}$\;
Compute $h(0)$ and $h(0.5)$\;
$UB \leftarrow \min(h(0),h(\alpha^*),h(0.5))$\;
Solve problem~\eqref{eq:prob1proof} and \eqref{eq:prob2proof} for $[0,\alpha^*]$\;
Solve problem~\eqref{eq:prob1proof} and \eqref{eq:prob2proof} for $[\alpha^*,0.5]$\;
Use Theorem~\ref{cor_L} to compute $L([0,\alpha^*])$\;
Use Theorem~\ref{cor_L} to compute $L([\alpha^*,0.5])$\;
$LB \leftarrow \min_{\I \in \mathcal{L}} L(\I)$\;
\While{$\mathcal{L} \neq \emptyset$}{
	Choose $\I' = \operatorname{argmin}_{\I \in \mathcal{L}} L(\I)$, with $\I'=[\alpha_1,\alpha_2]$\;
	$\mathcal{L} \leftarrow \mathcal{L} \setminus \I'$\;
	Choose $\alpha' \in \I'$ which defines $L(\I')$ (see the proof of Theorem~\ref{cor_L})\;
	Compute $h(\alpha')$\;
	Update $UB = \min\{UB,h(\alpha')\}$\;
	Solve problem~\eqref{eq:prob1proof} and \eqref{eq:prob2proof} for $[\alpha_1,\alpha']$\;
	Solve problem~\eqref{eq:prob1proof} and \eqref{eq:prob2proof} for $[\alpha',\alpha_2]$\;
	Use Theorem~\ref{cor_L} to compute $L([\alpha_1,\alpha'])$\;
	Use Theorem~\ref{cor_L} to compute $L([\alpha',\alpha_2])$\;
	\If {$L([\alpha_1,\alpha']) < UB$ and $[\alpha_1,\alpha'] \cap \mathcal{A} \neq \emptyset$}
		{$\mathcal{L} \leftarrow \mathcal{L} \cup \{[\alpha_1,\alpha']\}$}
	\If {$L([\alpha',\alpha_2]) < UB$ and $[\alpha',\alpha_2] \cap \mathcal{A} \neq \emptyset$}
		{$\mathcal{L} \leftarrow \mathcal{L} \cup \{[\alpha',\alpha_2]\}$}
	$LB \leftarrow \min_{\I \in \mathcal{L}} L(\I)$\;
}
\Return{$UB$}
\end{algorithm}

\section{Enumerative Algorithm}
\label{sec:algo:general}

Let us consider the set of feasible solutions to the deterministic combinatorial optimization problem as an ordered set, $\X=(x_1,\ldots,x_r)$, which we assume to know; we further explain how to compute $\X$ in Section~\ref{sec:computingX}. Let us reformulate problem \eqref{eq:minmaxmin_dU} as 
\begin{equation}
\label{eq:M3problem}
\min_{\xx\in\X^k}cost(\xx),
\end{equation}
where $cost(\xx)$ denotes the max-min cost of solution $\xx$, that is
\begin{equation}
\label{eq:costcomputation}
cost(\xx)=\max_{c\in \U} \min_{j\in[k]} c^\top \x{j}.
\end{equation}
Recall that we show in \eqref{eq:minmaxproblem} that \eqref{eq:costcomputation} can be reformulated as a linear program
\begin{equation}
 \label{eq:convexcost}
 cost(\xx)= \max_{c,z}\left\{ z : c\in \Uc, z \leq \sum_{i\in[n]} c_i \x{j}_i, \forall j\in [k]\right\}.
\end{equation}
The algorithm described in this section enumerates over all non-symmetric $k$-tuples $\xx\in\X^k$, using upper and lower bounds to prune part of the $k$-tuples and to avoid computing $cost(\xx)$ for all $k$-tuples. We also introduce the concept of \textit{resistance} to enumerate even less elements of $\X^k$. The pseudo-code is provided in Algorithm~\ref{algo:generalBB} for the case $k=2$. Throughout the section, we denote by $\permx$ the vector $(d_ix_i, i\in [n])$, sorted such that $\permx_1\geq \permx_2 \geq \cdots \geq \permx_n$.

\begin{algorithm}
\SetKwInOut{Return}{Return}
\caption{Enumerative algorithm (EA) illustrated for $k=2$.\label{algo:generalBB}}
\label{algo:enum2}
Let $UB$ be the initial upper bound from \eqref{eq:UB}\;
Compute $\X$\; \label{alg:genX}
Compute $lb(x)$ for each $x\in \X$ \label{ref:computelb}\;
\Repeat{$UB$ is not updated}
{
  Compute $\gamma^q(x)$ for each $x\in \X$\;
  Construct the partition $\X=\bigcup\limits_{\omega\in[q\cdot \Gamma]}\X_\omega^q$\;
  Let $r_\omega=|\X_\omega^q|$ for each $\omega$\;
  \ForEach{$\omega_1$ in $\{q\cdot\Gamma,q\cdot\Gamma-1,\ldots,1\}$}
  {
    \ForEach{$s_1$ in $\{1,\ldots,r_{\omega_1}\}$}
    {
      \ForEach{$\omega_2$ in $\{q\cdot\Gamma,q\cdot\Gamma-1,\ldots,q\cdot\Gamma+1-\omega_1\}$}
      {
	\lIf{$\omega_1 = \omega_2$} 
	  {
	  $s_2^{first}= s_1+1$
	  }
	  \lElse
	  {
	  $s_2^{first} = 1$
	  }
	\ForEach{$s_2$ in $\{s_2^{first},\ldots,r_{\omega_2}\}$}
	{
      $\xx=(\x{1},\x{2}) \leftarrow (x_{s_1},x_{s_2})$\;
	  Compute $LB_1(\xx)=\min\left(lb(\x{1}),{lb}(\x{2})\right)$ \label{algo:bound1} \label{algo:lb}\;
	  \If{$LB_1(\xx) > UB$} 
	  {
	    \textbf{continue}
	  }
	  \Else
	  {
	  Compute $LB_2(\xx)$ using a greedy algorithm \label{algo:bound2}\;
	    \If{$LB_2(\xx) > UB$}{ 
	      \textbf{continue}
	      }
	    \Else{
	      Compute $cost(\xx)$ by solving \eqref{eq:convexcost} \label{algo:cost}\;
	      \If{$cost(\xx)< UB$}
	      {
	      $UB \leftarrow cost(\xx)$\;
	      \textbf{break all for-loops} \label{algo:break}\;
	      }
	    }      
	}
      }
    }
    $\X\leftarrow \X\setminus \{x_{r_1}\}$\;
  }
  }
  }
\Return{the $k$-tuple with minimum cost}
\end{algorithm}

\subsection{Upper Bounds} 
Clearly an optimal solution of the classical robust Problem \eqref{eq:minmax} gives an upper bound for Problem \eqref{eq:minmaxmin_dU} since in the latter problem choosing the classical robust solution for each of the $k$ solutions we obtain the same objective value for both problems. As starting upper bound $UB$ on the optimal solution cost we choose
 \begin{equation}
\label{eq:UB}
 UB=\min(rob\_opt,heur),
\end{equation}
where $rob\_opt$ is the optimal value of the classical robust problem \eqref{eq:minmax}, obtained using the iterative algorithm from~\cite{BertsimasS03}, and $heur$ is the solution obtained by the local search algorithm from Section~\ref{sec:local}. The upper bound $UB$ is improved when a better feasible solution is found in the course of the algorithm.

\begin{observation}
\label{obs:LB0}
We only need to enumerate solutions $x\in \X$ with $\hc^\top x < UB$, since otherwise for every scenario $c\in\Uc$ we have $c^\top x\ge UB$ and therefore adding $x$ to a solution never improves the 
current upper bound. 
\end{observation}

\subsection{Lower Bounds}
\label{sec:lb}
As computing $cost(\xx)$ is time-consuming, we avoid computing its value exactly for many $k$-tuples and calculate instead two lower bounds, denoted by $LB_1(\xx)$ and $LB_2(\xx)$, defined below. 
Every time the cost of a $k$-tuple must be computed, we first compute $LB_1(\xx)$, which is done in $\is(k)$. If $LB_1(\xx)<UB$, then we compute $LB_2(\xx)$, requiring $\is(k\Gamma)$ steps. Only if $LB_2(\xx)<UB$ we compute $cost(\xx)$. We notice that these two bounds do not converge to $UB$, as is the case in many branch-and-bound algorithms. Hence, if the algorithm stops due to the time limit, it only returns a feasible solution to the problem, the remaining optimality gap being meaningless.

Each of the above bounds is derived by considering a particular scenario from $\Uc$. The first lower bound $LB_1(\xx)$ considers the scenario that assigns $\Gamma/k$ deviations per solution. To reduce the computational burden of computing $LB_1(\xx)$ to a minimum, once the enumeration has started, we compute in a pre-processing step (see line~\ref{ref:computelb} of Algorithm~\ref{algo:generalBB}) the cost of each solution $x$, by adding the $\lfloor \frac{\Gamma}{k}\rfloor$ largest deviations and a fraction of the $\lceil \frac{\Gamma}{k}\rceil$-th largest to the nominal costs of $x$. Formally we define
\begin{equation}
\nonumber
lb(x) = \hat c^{\top}x+\sum_{i=1}^{\lfloor \Gamma/k \rfloor}\permx_i + \left(\frac{\Gamma}{k} - \left\lfloor \frac{\Gamma}{k}\right\rfloor\right) \permx_{\lceil \Gamma/k \rceil}.
\end{equation}
Once $lb(x)$ is computed for each $x\in \X$, the lower bound can be obtained in $\is(k)$ through
\begin{equation}
\label{eq:LB1c2}
  LB_1(\xx) = \min_{j\in[k]}(lb(\x{j})).
\end{equation}
The second lower bound, denoted $LB_2(\xx)$ computes a good scenario greedily by taking the current solution $\xx$ and affecting the $\Gamma$ deviations sequentially to the solution having the smallest cost so far, which is the nominal cost plus the deviations already chosen. 

While the bounds significantly speed-up the computation of $cost(\xx)$, the cardinality of $\X^k$ is likely to be large. Fortunately, the majority of $k$-tuples in $\X^k$ can be discarded by using the concept of \emph{resistance} introduced next. 

\subsection{Resistance}
\label{sec:resistance}

Given an upper bound $UB$ and $q\in\N$, we define the discrete \emph{$q$-resistance} $\gamma^q(x)$ of any $x\in \X$ as the amount of deviation $\omega/q$ ($\omega \in \N$) that need to be affected to $\permx$ such that the cost of $x$ exceeds $UB$:
\begin{equation}
\label{eq:gammaq}
\gamma^q(x) = \min_{\omega\in \N}\left\{\omega : \hc^\top x + \sum_{i=1}^{\lfloor \omega/q \rfloor }\permx_{i} +(\omega/q-\lfloor \omega/q \rfloor)\permx_{\lceil \omega/q \rceil }
\geq UB\right\}. 
\end{equation}
Notice that if $\omega/q$ is integer, the third term of \eqref{eq:gammaq} vanishes. What is more, the value of $\gamma^q(x)$ is bounded above by $q\cdot\Gamma$.
To see this, suppose there exists a solution $x\in \X$ such that $\gamma^q(x) >q\cdot\Gamma$. Then, the cost of the $k$-tuple $\xx=(x,\ldots,x)$ satisfies $cost(\xx) < UB$. By definition of $\xx$, $cost(\xx)$ coincides with the classical robust value of $x$, denoted by $rob\_opt(x)$. Hence, $UB\leq rob\_opt(x) = cost(\xx) < UB$, which is a contradiction.

We show next that $\gamma^q(x)$ satisfies another crucial property.
\begin{lemma}
\label{lem:resistance}
 If $\xx\in \X^k$ with $\sum_{j\in [k]} \gamma^q(\x{j}) \leq q\cdot\Gamma$, then $cost(\xx) \geq UB$.
\end{lemma}
\begin{proof}
Let $\pi^x$ be the permutation of $[n]$ used to obtain $\permx$ from the vector $(d_ix_i, i\in [n])$, that is, $\permx_i= d_{\pi^x(i)}$.
For each $j\in[k]$, we define the vector $z^{(j)}$ as 
$$
z^{(j)}_i=\left\{
\begin{array}{ll}
1 &\mbox{ if }\pi^{\x{j}} \leq  \lfloor \gamma^q(\x{j}) /q \rfloor\\
\omega/q-\lfloor \omega/q \rfloor &\mbox{ if }\pi^{\x{j}} =  \lceil \gamma^q(\x{j}) /q \rceil\\
0 & 
\mbox{otherwise}
\end{array}
\right..
$$
From \eqref{eq:gammaq}, we see that $\sum_{i\in[n]}(\hc_i+z^{(j)}_i d_i)\x{j} \geq UB$. Next, we define $z^*_i=\max_{j\in[k]}z^{(j)}_i$ for each $i\in[n]$, and we have
\begin{equation}
\label{eq:gUB}
\sum_{i\in[n]}(\hc_i+z^*_i d_i)\x{j} \geq UB.
\end{equation}
Since $\sum_{j\in [k]} \gamma^q(\x{j}) \leq q\cdot \Gamma$, we have that $\sum_{i\in [n]} z^*_i\leq \Gamma$ so that 
\begin{equation}
\label{eq:belongsU}
(\hc_i+z_id_i, i\in [n])\in\Uc.
\end{equation}
Statements \eqref{eq:gUB} and \eqref{eq:belongsU} imply that $cost(\xx) \geq UB$.
\end{proof}

Thanks to the lemma above, we only have to enumerate solutions $\xx\in\X^k$ with $\sum_{j\in[k]} \gamma^q(\x{j}) > q \cdot \Gamma$. Let us define $\X^q_\omega=\{x\in \X: \gamma^q(x) = \omega\}$ for each $\omega \in \{0,\ldots,q\cdot\Gamma\}$. We have that $\X=\cup_{\omega = 0}^{q\cdot\Gamma}\X^q_\omega$. The following observation states that we do not need to consider the elements from $\X^q_0$.
\begin{observation}
\label{obs:X0}
 For any $x\in \X^q_0$ it holds $\hc^\top x \geq UB$.
\end{observation}
Following Lemma~\ref{lem:resistance} and Observations~\ref{obs:LB0} and~\ref{obs:X0}, we need to enumerate only the subset of all $k$-tuples defined by
$$
\X^{k,q}_\Gamma = \left\{\xx\in\X^k: \x{j}\in\X_{\omega_j}^q, \omega_j\in[q\cdot \Gamma], j\in[k], \sum_{j\in [k]} \omega_j > q\cdot\Gamma\right\}.
$$

Notice that the sets $\X_\omega^q$ can be updated every time a better upper bound is found, which explains the presence of \textbf{break} in line~\ref{algo:break} of Algorithm~\ref{algo:generalBB}. In particular, the cardinality $|\X_0^q|$ increases with $UB$. Hence, Observation~\ref{obs:X0} implies that restarting the enumeration (step~4 of Algorithm~\ref{algo:generalBB}) when improving $UB$ possibly leads to the removal of many elements of $\X$ at each restart.

\subsection{Handling Symmetry} 
The symmetry among the elements of $\X^k$ can be used to reduce the set of feasible solutions. Specifically, if $\xx$ and $\xx'$ are made of the same elements of $\X$, but listed in different orders, $cost(\xx)=cost(\xx')$. Hence, we focus in what follows on the set $\XX^k\subset\X^k$, defined as $\XX^k=\{(x_{s_1},\ldots,x_{s_k})\in \X^k: s_j<s_{j+1}, j=1,\ldots,k-1\}$, and we define $\XX^{k,q}_\Gamma$ analogously to Section~\ref{sec:resistance}.

\subsection{Computing $\X$}
\label{sec:computingX}

In this section we detail two algorithms to enumerate all elements of $\X'=\{x\in \X:\hc^\top x < UB\}$. Our first approach relies on an ad-hoc branch-and-bound algorithm that enumerates all feasible solutions to the problem
$$
\min_{x\in \X, \hc^\top x < UB} \hc^\top x,
$$
by branching iteratively on all variables, and collecting all leaves having accumulated $n$ branching constraints. To avoid exploring the full branch-and-bound tree (which would contain $2^n$ leaves for $\X=\{0,1\}^n$), the algorithm combines $UB$ with the bound provided by a relaxation to prune parts of the tree. Specifically, let $\X^{LP}$ be a formulation for $\X$, that is, a polytope such that $\X^{LP}\cap\{0,1\}^n=\X$, and let us introduce the branching constraints through the disjoint sets $O,Z\subseteq[n]$. At each node of the branch-and-bound tree, the algorithm solves the relaxation $\X^{LP}$ together with the branching constraints accumulated so far
\begin{equation}\label{eq:LP}\tag{LP}
\begin{aligned}
z^{LP}\;=\;\min \ & \sum_{i\in[n]} \hc_i x_i\\ 
\text{s.t. } & x \in \X^{LP} \\
& x_i = 0 &\forall i \in Z \\
& x_i = 1 &\forall i \in O
\end{aligned}
\end{equation}
Nodes of the tree are pruned either because \eqref{eq:LP} is infeasible or because $z^{LP}\geq UB$.

The above approach can be improved significantly for problems for which $\X$ can be enumerated through recursive algorithms, such as the knapsack problem, the shortest path problem, the traveling salesman problem, and spanning (Steiner) tree problems, among many others. In that situation, one can embed the constraint $\hc^\top x < UB$ in the recursive algorithms, allowing us to generate all elements of $\X'$ for problems of reasonable dimensions. Let us detail this approach for problem of finding a shortest path from $s$ to $t$ in an undirected graph $G$ with positive costs $\hc$, considering a Depth First Search (DFS). We execute first the Dijkstra algorithm from $t$ to compute the distance between $t$ and each node $v$ in the graph $d(v,t)$. Then, every time a node $v$ is discovered by the DFS along a path $P$, starting from $s$, we further consider its successors only if 
$
\sum_{i\in P}\hc_i + d(v,t) < UB.
$

We compare the two algorithms in our computational experiments.

\section{Computational Results}
\label{sec:experiments}

The aim of this section is two-fold. First, we assess the cost reduction offered by the min-max-min model \eqref{eq:minmaxmin_dU} when compared to the classical min-max robust model \eqref{eq:minmax}. Second, we evaluate in detail the numerical efficiency of the proposed exact and heuristic solution algorithms. Our experiments are carried out on a set of shortest path instances previously used by~\cite{hanasusanto2015k} and on randomly generated instances for the min-knapsack problem. In what follows, HKW denotes 
linearized compact formulation from Section~\ref{sec:problemref} (previously used in \cite{hanasusanto2015k}),
heur stands for the heuristic algorithm from Section~\ref{sec:local},
BB denotes the branch-and-bound Algorithm~\ref{alg:exact_2_adapt} from Section~\ref{sec:bb}, EA denotes the enumeration Algorithm~\ref{algo:generalBB} from Section~\ref{sec:algo:general}.
 
All LPs and MIPs involved are solved by CPLEX version~$12.6$. Algorithms from Section~\ref{sec:algo:convex} and Section~\ref{sec:algo:general} are implemented in $C$++ and julia, respectively. The local search uses a processor i5-3470 running at 3.2 Ghz while the exact algorithms use a processor Intel X5460 running at 3.16 GHz, respectively. Note that for the experiments in \cite{hanasusanto2015k}, Gurobi Optimizer 5.6 was used, but no details are provided on their computer speed.
All solution times are reported in seconds.

\subsection{Shortest Path Problem (SP)}
\subsubsection{Instances}
We use the shortest path instances presented in \cite{hanasusanto2015k}. Each instance is described by three parameters: the number of nodes $|V|\in\{15 + 5i : i\in[7]\}$ of the underlying graph $G=(V,E)$, the number $k\in\{2,3,4\}$ of candidate solutions, and the parameter $\Gamma\in\{3,6\}$ which specifies the size of the uncertainty set. 
For each parameter combination $100$ instances are randomly generated, which results in 4200 instances in total.
For the heuristic algorithm presented in Section \ref{sec:local}, we set the time limit of the $x$-step to $300$ seconds. For all $4200$ problem instances this time limit was only met $6$ times. For the exact approaches, we set a time limit of 7200 seconds for each experiment, as in \cite{hanasusanto2015k}.

\subsubsection{Solution Costs}

We present in Table~\ref{fig:costred} the cost reductions obtained by each algorithm compared to the cost provided by the robust model \eqref{eq:minmax}. Specifically, for each algorithm $A$, we report
\begin{equation}
\label{eq:cost_red}
\mbox{cost\_red}(A)=100\times \frac{\opt\eqref{eq:minmax} - \opt(A)}{\opt\eqref{eq:minmax}}.
\end{equation}
Notice that the three exact algorithms (BB, EA, and HKW) leave some instances unsolved (see the next section for details), explaining why two approaches $A$ and $A'$ for a given value of $k$ may lead to different values for cost\_red$(A)$. We see from the table that EA always obtains the highest value, followed closely by the three other approaches. In particular, the table underlines that the feasible solutions calculated by HKW are of good quality, even though the algorithm cannot prove optimality for most instances and often finishes with optimality gaps greater than 5\% (see Tables~\ref{tab_results_exact} and~\ref{tab_results_exact_k3}). The quality of the heuristic is also very good, following closely the results of HKW, even improving over the latter in some cases ($k=4$, $\Gamma=6$ and $|V|\geq 40$). Last, the table illustrates the decreasing benefit of increasing the value of $k$. While the cost reduction is important for $k=2$, the subsequent improvements are much smaller, in particular for $\Gamma = 3$.

\begin{table}
\centering
\begin{tabular}{c|c|r|r|r|r||r|r|r||r|r}
    && \multicolumn{4}{c||}{$k=2$} & \multicolumn{3}{c||}{$k=3$} & \multicolumn{2}{c}{$k=4$} \\	
$\Gamma$&$|V|$&	EA&	BB&	heur&	HKW&	EA&	heur&	HKW&	heur&	HKW\\
\hline
\multirow{8}{*}{3}&20	& 7.6	& 7.6	& 7.5	& 7.6	& 9.2	& 9.0	& 9.2	& 9.3	& 9.5\\
&25	& 8.6	& 8.6	& 8.6	& 8.6	& 10.7	& 10.6	& 10.7	& 11.0	& 11.3\\
&30	& 8.7	& 8.7	& 8.6	& 8.7	& 11.3	& 11.1	& 11.3	& 11.9	& 12.2\\
&35	& 9.0	& 9.0	& 8.9	& 9.0	& 12.0	& 12.0	& 12.0	& 12.8	& 13.1\\
&40	& 9.0	& 9.0	& 9.0	& 9.0	& 12.1	& 12.0	& 12.0	& 12.9	& 13.2\\
&45	& 8.6	& 8.6	& 8.5	& 8.6	& 11.8	& 11.7	& 11.8	& 12.9	& 13.1\\
&50	& 8.6	& 8.6	& 8.4	& 8.6	& 11.8	& 11.7	& 11.7	& 13.0	& 13.1\\
\cline{2-11}
&AVG	& 8.5	& 8.5	& 8.4	& 8.5	& 10.8	& 10.7	& 10.8	& 11.2	& 11.5\\
\hline
\hline
\multirow{8}{*}{6}&20	& 6.7	& 6.7	& 6.6	& 6.7	& 10.2	& 10.1	& 10.2	& 11.2	& 11.4\\
&25	& 8.2	& 8.2	& 8.1	& 8.2	& 12.2	& 12.1	& 12.2	& 13.6	& 13.8\\
&30	& 8.6	& 8.6	& 8.5	& 8.6	& 13.0	& 12.9	& 13.0	& 15.0	& 15.1\\
&35	& 9.2	& 9.2	& 9.1	& 9.2	& 13.9	& 13.8	& 13.9	& 16.3	& 16.3\\
&40	& 9.5	& 9.5	& 9.4	& 9.5	& 14.3	& 14.2	& 14.2	& 16.7	& 16.6\\
&45	& 9.8	& 9.7	& 9.6	& 9.7	& 14.5	& 14.4	& 14.4	& 17.0	& 16.9\\
&50	& 9.8	& 9.7	& 9.6	& 9.7	& 14.5	& 14.4	& 14.4	& 17.0	& 16.9\\
\cline{2-11}
&AVG	& 8.2	& 8.2	& 8.1	& 8.2	& 12.3	& 12.3	& 12.3	& 14.0	& 14.1
	\end{tabular}
\caption{Cost reduction cost\_red$(A)$ for each algorithm $A$ for the SP.\label{fig:costred}}
\end{table}

\subsubsection{Solution Times}

We first present the solution times of the heuristic algorithm, before investigating in detail the results of the three exact algorithms. The average run times of the heuristic are reported in Table~\ref{tab_runtime_heuristic} for each group of 100 instances. The vast majority of instances are solved within a minute, often in a few seconds. 

\begin{table}[htb]
\centering
\begin{tabular}{c|r|r|r||r|r|r}
		& \multicolumn{3}{c||}{$\Gamma =3$} & \multicolumn{3}{c}{$\Gamma =6$} \\ \cline{2-7}
$|V|$ 	& \multicolumn{1}{c|}{$k=2$} & \multicolumn{1}{c|}{$k=3$} & \multicolumn{1}{c||}{$k=4$} & \multicolumn{1}{c|}{$k=2$} & \multicolumn{1}{c|}{$k=3$} & \multicolumn{1}{c}{$k=4$} \\ \hline
20   	&   $0.1$  &  $0.1$   &  $0.2$   &  $0.2$   & $0.2$    &  $0.2$ \\
25   	&   $0.2$  &  $0.3$   &  $0.2$   &  $0.3$   & $0.6$    &  $0.6$ \\
30   	&   $0.4$  &  $0.6$   &  $0.5$   &  $0.7$   & $1.3$    &  $2.3$ \\
35   	&   $0.7$  &  $1.1$   &  $1.1$   &  $1.6$   & $4.2$    &  $7.5$ \\
40   	&   $1.0$  &  $1.3$   &  $1.6$   &  $3.3$   & $7.8$    &  $19.9$ \\
45   	&   $1.7$  &  $2.5$   &  $3.5$   &  $6.3$   & $18.8$    &  $39.7$ \\
50   	&   $2.1$  &  $3.0$   &  $4.9$   &  $12.1$   & $29.6$    &  $73.5$
\end{tabular}
\caption{Computation times of heur in seconds for the SP.}
\label{tab_runtime_heuristic}
\end{table}

\begin{table}
\centering
\begin{tabular}{c|c||r|r H||r|r|r||r|r}
& & \multicolumn{3}{c||}{Time} & \multicolumn{3}{c||}{\% Solved} & \multicolumn{2}{c}{\% Gap} \\
$\Gamma$ & $|V|$ & BB & EA & heu. & BB & EA &	HKW & BB & HKW \\ \hline
\multirow{7}{*}{3}	
& 20   	&  9.9$\pm$8.1& 0.2$\pm$0.8   & 0.1$\pm$0.0 & $100$ & 100 & 100   &  $0.00$ & $0.00$ 		\\
& 25   	&  19.5$\pm$13.4 & 0.2$\pm$0.7   & 0.1$\pm$0.1 & 100 & 100 & 99  &  $0.00$ &  $3.68$ 		\\
& 30   	&  53.4$\pm$57.4 & 0.9$\pm$3.0   & 0.3$\pm$0.2   &100 & 100 & 69 &  $0.00$ &  $5.69$		\\
& 35   	&  92.4$\pm$103.6 & 0.9$\pm$1.0  & 0.5$\pm$0.5  & 100 & 100 & 17  &  $0.00$ &  $5.70$	 	\\
& 40   	&  168.2$\pm$176.3 & 2.0$\pm$2.9  & 0.8$\pm$0.8 & 100 & 100 & 6  &  $0.00$ &  $5.96$	 	\\
& 45   	&  385.3$\pm$429.6 & 8.0$\pm$16.4 & 1.8$\pm$2.3  & 100 & 100 & 0 &  $0.00$ &  $6.48$ 		\\
& 50   	&  654.9$\pm$700.3 & 12.9$\pm$17.9  & 2.5$\pm$2.5 & 100 & 100 & 0  &  $0.00$ &  $6.75$ 		\\ \hline
\multirow{7}{*}{6}  
& 20   	&  186.1$\pm$832.8 & 0.2$\pm$0.8  & 0.1$\pm$0.2    & 99 & 100 & 100 &  $8.0e$-$6$  & $0.00$ 			\\
& 25   	&  153.2$\pm$358.2 & 0.6$\pm$0.6  & 0.4$\pm$0.8 & 100 & 100 & 100  &  $0.00$ & $8.00$ 			\\
& 30   	&  651.8$\pm$930.9 & 2.8$\pm$3.7  & 1.1$\pm$1.4   & 100 & 100 & 67 &  $0.00$ & $10.72$			\\
& 35   	&  1984.2$\pm$2075.2 & 9.0$\pm$10.9  & 3.5$\pm$4.8  & 92 & 100 & 16  &  $7.3e$-$3$ & $10.76$	\\
& 40   	&  3484.7$\pm$2723.7 & 26.3$\pm$34.0  & 9.3$\pm$14.1  & 72 & 100 & 5 &  $1.8e$-$2$ & $11.29$		\\
& 45   	&  4897.6$\pm$2538.6 & 117.7$\pm$312.5 & 21.3$\pm$30.5  & 54 & 100 & 0 &  $1.2e$-$1$ & $11.79$ 	\\
& 50   	&  6188.8$\pm$1951.9 & 318.6$\pm$577.2  & 50.2$\pm$114.7  & 30 & 99 & 0 &  $2.3e$-$1$ & $12.31$ 
\end{tabular}
\caption{Comparison of HKW, EA, and BB for $k=2$ for the SP.}
\label{tab_results_exact}
\end{table}

We next turn to the exact methods and provide a more detailed presentation of the results. We present in Table~\ref{tab_results_exact} a comparison of the exact solution times of BB, EA using the DFS strategy described in Section~\ref{sec:computingX}, and HKW for $k=2$. Computation times include the means and standard deviations over the 100 instances of each group (unsolved instances count for 7200 seconds). In the last two columns we report the average percental gap between upper and lower bound which was reached after the time limit of 7200 seconds. We see that all instances were solved to optimality by EA in a few seconds, and to near optimality by BB. It turns out that the instances of the smaller uncertainty set ($\Gamma = 3$) are easier to solve by our methods. For some instances of the larger uncertainty set the time limit was reached by BB. However, the remaining gap between upper and lower bound reported by BB is reasonably small as reported in Table~\ref{tab_results_exact}. 
Notice that algorithm EA returns no optimality gap when it fails to solve the problem to optimality.

Table~\ref{tab_results_exact} clearly shows that both EA and BB outperform HKW by orders of magnitude. Further, EA is also much faster than BB. Notice, however, that EA requires large amounts of memory: the only unsolved instance by EA failed because of a memory hit, using a computer with 48 GB of memory. In fact, many large instances require more than 20 GB of memory, while BB can handle all instances with a few GBs. The difference in memory consumption is due to the large cardinality of the set $\X$ and the small number of nodes explored by BB. These two aspects are further investigated in Figure~\ref{fig:X}.
Figure~\ref{fig:X:a} shows that the number of solutions $x$ that satisfy $\hc^\top x<UB$ increases nearly exponentially with $|V|$, reaching roughly $8\times10^8$ for $|V|=50$ and $\Gamma=6$. In contrast, 
Figure~\ref{fig:X:nodes} shows that the number of nodes explored by BB does not seem impacted by $|V|$ and revolves around 100 nodes, while Figure~\ref{fig:X:gap} even shows the root gaps decrease with $|V|$. The three charts of Figure~\ref{fig:X} also indicate that $\Gamma=6$ leads to significantly harder instances than $\Gamma=3$. 

Figure~\ref{fig:costcomp} investigates the effect of the \emph{resistance} and the lower bounds, described in Sections~\ref{sec:resistance} and~\ref{sec:lb}, respectively. Recall that, without using resistance, the number of $2$-tuples enumerated should be $\frac{|\X|(|\X|-1)}{2}$. Hence, $|\X|\sim 100$ should lead to $5\times 10^4$ 2-tuples, which is far from the results shown on the plain boxes from Figure~\ref{fig:costcomp}. For instance, consider the case $\Gamma=6$ and $|V|=50$. Figure~\ref{fig:X:a} shows that for nearly 95\% of the instances, $|\X|\geq 10^5$. Yet, Figure~\ref{fig:costcomp:6} shows that more than 95\% of the instances enumerate at most $10^5$ 2-tuples. Regarding the interest of the lower bounds, the dotted boxplots from Figure~\ref{fig:costcomp} show that the number of 2-tuples for which $cost(x_{s_1},x_{s_2})$ is actually computed is between half and one order of magnitude less than the number of 2-tuples enumerated.

\begin{figure}
	\centering
	\subfigure[$|\X|$]{\label{fig:X:a}
	\includegraphics[width=7cm]{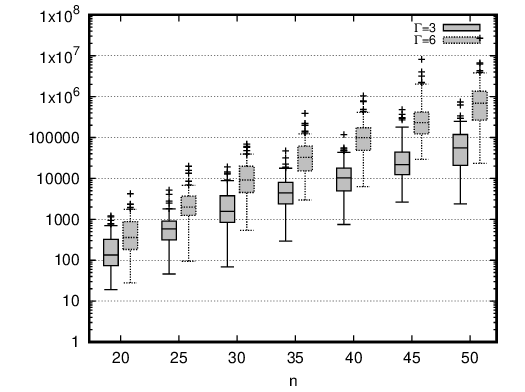}}
		\subfigure[nodes explored by BB]{\label{fig:X:nodes}
	\includegraphics[width=7cm]{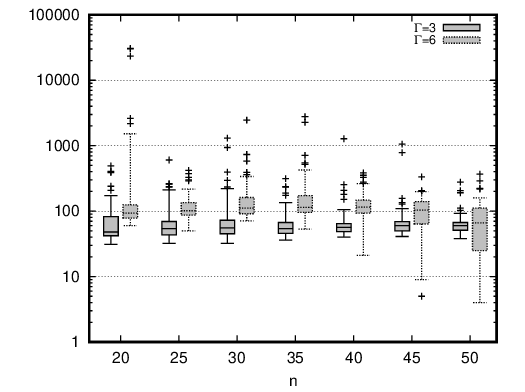}
	}
		\subfigure[root gap]{\label{fig:X:gap}
	\includegraphics[width=7cm]{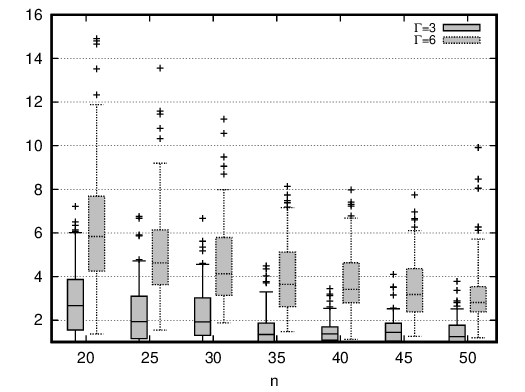}
	}
\caption{Size of $\X$ for the SP generated in Step~\ref{alg:genX} of Algorithm~\ref{algo:generalBB}, number of nodes explored by BB, in logarithmic scale, and the root gap of BB. The box contains 50\% of the data samples, cut by the median. The whiskers extend the box to cover 95\% of the samples.\label{fig:X}}
\end{figure}
\begin{figure}
	\centering
	\subfigure[$\Gamma=3$]{
	\includegraphics[width=7cm]{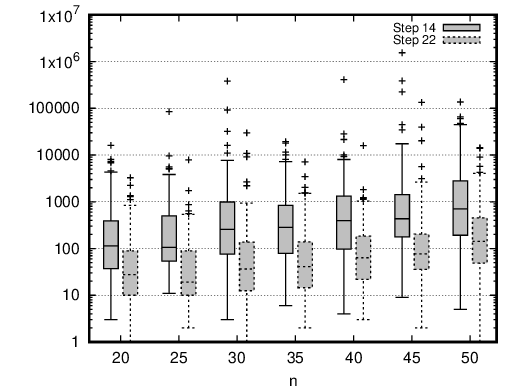}}
		\subfigure[$\Gamma=6$]{\label{fig:costcomp:6}
	\includegraphics[width=7cm]{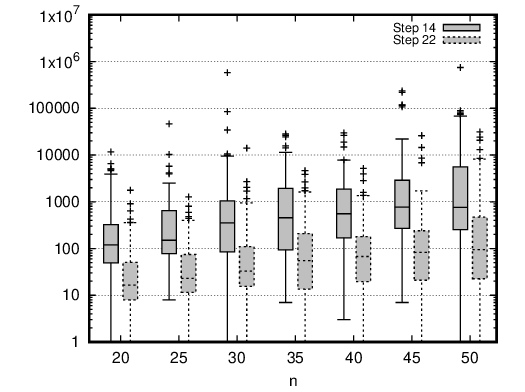}
	}
\caption{Number of $2$-tuples $(x_{s_1},x_{s_2})$ handled by Algorithm~\ref{algo:generalBB} at steps~\ref{algo:lb} and~\ref{algo:cost}, in logarithmic scale for the SP.\label{fig:costcomp}}
\end{figure}

\begin{table}
\centering
\begin{tabular}{c|c H H H||r H||r|r||r}
& & \multicolumn{3}{c||}{} & \multicolumn{2}{c||}{Time EA} & \multicolumn{2}{c||}{\% Solved} & \% Gap \\
$\Gamma$ & $|V|$ & EA & HKW & heu. &  & heu. & EA &	HKW & HWK\\ \hline
\multirow{7}{*}{3} & 
20 & 14.9166 & 14.9166 & 14.9417 & 80.7$\pm$261.0 & 0.11 & 100 & 97 & 1.60 \\
&25 & 14.9103 & 14.9133 & 14.9417 & 249.3$\pm$1027.0 & 0.28 & 98 & 31 & 1.14 \\
&30 & 14.6429 & 14.6447 & 14.6703 & 621.7$\pm$1539.8 & 5.81 & 97 & 6 & 1.71 \\
&35 & 14.4564 & 14.4621 & X & 741.0$\pm$1795.5 & 10.58 & 95 & 0 & 2.23 \\
&40 & 14.3203  & 14.3257  & X & 1232.8$\pm$2322.8 & 10.58 & 90 & 0 & 2.59 \\
&45 & 14.2275  & 14.2355  & X & 1871.9$\pm$2832.0 & 10.58 & 82 & 0 & 3.14 \\
&50 & 14.1522  & 14.1679 & X & 2210.7$\pm$2817.2 & 10.58 & 80 & 0 & 3.44 \\
\hline
\multirow{7}{*}{6} & 
20 & 16.4442 & 16.4442 & 16.4499 & 49.2$\pm$172.6 & 0.19 & 100 & 97 & 4.06 \\
&25 & 16.4114 & 16.4123 & 16.4217 & 249.4$\pm$1025.0 & 5.89 & 99 & 38 & 3.52 \\
&30 & 16.0752 & 16.0773 & 16.09271 & 593.4$\pm$1455.9 & 41.90 & 97 & 6 & 4.54 \\
&35 & 15.8744 & 15.8827 & 16.09271 & 1133.6$\pm$1964.7 & 41.90 & 96 & 0 & 5.58 \\
&40 & 14.6961 & 15.7050 & 16.09271 & 2466.6$\pm$2646.6 & 41.90 & 82 & 0 & 6.19 \\
&45 & 15.5472 & 15.5593 & 16.09271 & 4529.1$\pm$2846.3 & 41.90 & 53 & 0 & 6.92 \\
&50 & 15.4500 & 15.4751 & 16.09271 & 6021.5$\pm$2122.8 & 41.90 & 29 & 0 & 7.55 \\
\end{tabular}
\caption{Comparison of HKW and EA for $k=3$ for the SP.}
\label{tab_results_exact_k3}
\end{table}

We present in Table~\ref{tab_results_exact_k3} a comparison of the exact solution times of EA and HKW for $k=3$. While EA cannot solve all instances during the time limit, it still outperforms HKW significantly, solving 642 ($\Gamma=3$) and 556 ($\Gamma=6$) instances to optimality (out of 700), instead of 134 and 141 for HKW. 

\begin{table}
\centering
\begin{tabular}{c||r|r||r|r}
& \multicolumn{2}{c||}{$\Gamma = 3$} & \multicolumn{2}{c}{$\Gamma = 6$} \\
$|V|$ & Time & \% Solved & Time & \% Solved\\
\hline
20 & 3.1$\pm$3.2  & 100 & 3.3$\pm$3.7 & 100 \\
25 & 16.2$\pm$15.4 & 100 & 16.3$\pm$15.9 & 100 \\
30 & 100.5$\pm$114.1 & 100 & 105.1$\pm$116.6 & 100 \\
35 & 347.1$\pm$362.4 & 100 & 358.5$\pm$372.1 & 100 \\
40 & 1148.2$\pm$1146.8 & 99 & 1174.0$\pm$1163.0 & 99 \\
45 & 3310.7$\pm$2281.9 & 85 & 3348.0$\pm$2288.9 & 85 \\
50 & 5374.4$\pm$2283.9 & 45 & 5402.5$\pm$2263.2 & 45 \\
\end{tabular}
\caption{Results of EA generating $\X$ using a vanilla branch-and-bound algorithm for $k=2$ and $\Gamma=3$ for the SP.}
\label{tab:LPbb}
\end{table}

The results presented for EA so far have relied on the DFS strategy to iterate through the set $\X$. To understand whether the alternative LP-based branch-and-bound algorithm is a realistic way to iterate through $\X$, we have also coded a vanilla version of that branch-and-bound algorithm for the shortest path problem. The latter is coded in julia, using package JuMP, and does not implement advanced warm-starts when processing a new node. The results presented in Table~\ref{tab:LPbb} indicate that this strategy is orders of magnitude slower than the DFS. Yet, it is able to solve nearly all instances having less than 50 nodes during the time limit. Nearly 100\% of the time is spent in the generation of $\X$.

\subsection{Knapsack Problem (KP)}

\subsubsection{Instances}
Since the paper has studied minimization optimization problems so far, we consider next a minimization version of the knapsack problem defined as
$$
\min \left\{\sum_{i\in [n]} c_i x_i : \sum_{i\in [n]} w_i x_i \geq W, \; x\in \{0,1\}^n\right\}.
$$
For each dimension $n$ the costs $c_i$ and the weights $w_i$ were drawn from a uniform distribution on $\left\{ 1,\ldots ,100\right\}$. The knapsack capacity $W$ was set to $35\%$ of the sum of all weights. For each knapsack instance we generate a budgeted uncertainty set with mean vector $\hat c=c$ and a random deviation vector $d$ where each $d_i$ is drawn uniformly in $\left\{ 1,\ldots ,c_i\right\}$. Each instance is described by three parameters: the number of items $|n|\in\{50i : i\in[4]\}$, the number $k\in\{2,3,4\}$ of candidate solutions, and the parameter $\Gamma\in\{3,6\}$ which specifies the size of the uncertainty set. 
For each parameter combination $10$ instances are randomly generated, which results in 240 instances in total.
We set a time limit of 3600 seconds for each experiment.
The approach HKW mentioned in the following reports on our implementation of the linearized formulation described in Section~\ref{sec:algo:convex}.

\subsubsection{Solution Costs}

We present in Table~\ref{fig:knap:costred} the application of formula~\eqref{eq:cost_red} to the knapsack instances. As for the shortest path, the three exact algorithms (BB, EA, and HKW) leave some instances unsolved, see the next section for details. The table shows that the cost reductions for the knapsack problem are much smaller than those obtained for the shortest path problem, and these reductions decrease with the size of the instances. We also see that the heuristic solutions are usually better than the best solutions returned by HKW, which is not surprising given the high gaps returned by the latter (see Table~\ref{knap_tab_results_exact}).

\begin{table}
\color{black}
\centering
\begin{tabular}{c|c|r|r|r|r||r|r|r||r|r}
    && \multicolumn{4}{c||}{$k=2$} & \multicolumn{3}{c||}{$k=3$} & \multicolumn{2}{c}{$k=4$} \\	
$\Gamma$&$|V|$&	EA&	BB&	heur&	HKW&	EA&	heur&	HKW&	heur&	HKW\\
\hline
\multirow{4}{*}{3}&50 & 2.0 & 2.0 & 1.9 & 1.9 & 2.5 & 2.4 & 2.1 & 2.5 & 2.3 \\
&100 & 2.0 & 2.0 & 2.0 & 1.6 & 2.0 & 2.4 & 1.4 & 2.5 & 1.2 \\
&150 & 1.2 & 1.2 & 1.2 & 0.3 & 1.2 & 1.5 & 0.0 & 1.6 & 0.0 \\
&200 & 0.5 & 0.5 & 0.5 & 0.0 & 0.5 & 0.7 & 0.0 & 0.8 & 0.0 \\
\cline{2-11}
&AVG & 1.4 & 1.4 & 1.4 & 1.0 & 1.5 & 1.7 & 0.9 & 1.9 & 0.9 \\
\hline
\hline
\multirow{4}{*}{6}&
50 & 1.5 & 1.5 & 1.4 & 1.5 & 2.2 & 2.1 & 1.6 & 2.4 & 1.4 \\
&100 & 2.4 & 2.5 & 2.4 & 2.0 & 2.4 & 3.2 & 1.4 & 3.4 & 1.0 \\
&150 & 2.0 & 2.0 & 2.0 & 0.4 & 2.0 & 2.4 & 0.7 & 2.6 & 0.1 \\
&200 & 1.1 & 1.1 & 1.1 & 0.0 & 1.1 & 1.3 & 0.0 & 1.4 & 0.0 \\
\cline{2-11}
 &AVG & 1.8 & 1.8 & 1.7 & 1.0 & 1.9 & 2.3 & 0.9 & 2.5 & 0.6 \\
	\end{tabular}
\caption{Cost reduction cost\_red$(A)$ for each algorithm $A$ and value of $k$ for the KP.\label{fig:knap:costred}}
\end{table}

\subsubsection{Solution Times}

We first present the solution times of the heuristic algorithm. The average run times of the heuristic are reported in Table~\ref{knap_tab_runtime_heuristic} for each group of 100 instances. As before, the vast majority of instances are solved within a minute, often in a few seconds. 

\begin{table}[htb]
\color{black}
\centering
\begin{tabular}{c|r|r|r||r|r|r}
		& \multicolumn{3}{c||}{$\Gamma =3$} & \multicolumn{3}{c}{$\Gamma =6$} \\ \cline{2-7}
$|V|$ 	& \multicolumn{1}{c|}{$k=2$} & \multicolumn{1}{c|}{$k=3$} & \multicolumn{1}{c||}{$k=4$} & \multicolumn{1}{c|}{$k=2$} & \multicolumn{1}{c|}{$k=3$} & \multicolumn{1}{c}{$k=4$} \\ \hline
50 & 0.0 & 0.1 & 0.1 & 0.0 & 0.1 & 0.1 \\
100 & 0.1 & 0.3 & 1.0 & 0.2 & 0.6 & 3.9 \\
150 & 0.2 & 1.5 & 8.9 & 0.3 & 2.7 & 56.8 \\
200 & 0.4 & 1.3 & 15.2 & 0.5 & 3.9 & 34.8 \\
\end{tabular}
\caption{Computation times of heur in seconds for the KP.}
\label{knap_tab_runtime_heuristic}
\end{table}

Table~\ref{knap_tab_results_exact} compares HKW, BB, and EA on the knapsack instances for $k=2$. The following observations arise from results presented in the table. First, HKW is not able to solve \emph{any} of the instances, often resulting in large optimality gaps. The inefficiency of HKW is explained by its extremely weak continuous relaxation, the optimal solution of which is 0 for all instances. Second, EA is only able to solve the smallest instances. For 100 items or more, it cannot finish the first step of the algorithm (enumerating the solutions). As it turns out, the knapsack problem is particularly difficult for EA because that problem lacks strong constraints on the structure of the feasible solutions, so that there exists plenty of similar solutions that cannot be removed before the algorithm starts. Third, as for the shortest path, BB can solve to optimality most of the instances, ending with small optimality gaps for the unsolved ones. The efficiency of BB is mainly due to its tight root gap, as further illustrated in Figure~\ref{fig:BBrootgap}.

\begin{table}
\color{black}
\centering
\begin{tabular}{c|c||r|r||r|r|r||r|r}
& & \multicolumn{2}{c||}{Time} & \multicolumn{3}{c||}{\% Solved} & \multicolumn{2}{c}{\% Gap} \\
$\Gamma$ & $|V|$ & BB & EA & BB & EA &	HKW & BB & HKW \\ \hline
\multirow{4}{*}{3} &	
50 & 10.9$\pm$6.7 & 73.9$\pm$73.6 & 100 & 100 & 0 & 0 & 14.0 \\
&100 & 32.3$\pm$29.9 & 3600 & 100 & 0 & 0 & 0 & 69.9 \\
&150 & 411.7$\pm$1063.5 & 3600 & 90 & 0 & 0 & 2.8$e$-2 & 90.0 \\
&200 & 636.6$\pm$1037.2 & 3600 & 90 & 0 & 0 & 1.2 & 96.0 \\
\hline
\multirow{4}{*}{6} &50 & 69.6$\pm$124.8 & 259.8$\pm$328.5 & 100 & 100 & 0 & 0 & 15.2 \\
&100 & 351.9$\pm$811.5 & 3600 & 100 & 0 & 0 & 0 & 71.1 \\
&150 & 540.9$\pm$550.1 & 3600 & 100 & 0 & 0 & 0 & 90.2 \\
&200 & 1905.9$\pm$1244.4 & 3600 & 70 & 0 & 0 & 1.8 & 96.0 \\	
\end{tabular}
\caption{Comparison of HKW, EA, and BB for $k=2$ for the KP.}
\label{knap_tab_results_exact}
\end{table}

\begin{figure}
	\centering
	\includegraphics[width=7cm]{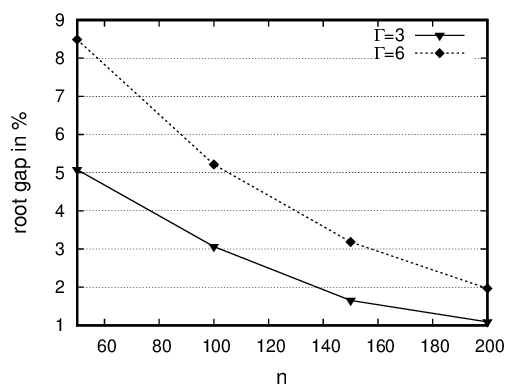}
\caption{Root gaps for the KP.\label{fig:BBrootgap}}
\end{figure}

For $k=3$, EA is able to solve 7 out of 10 instances for $\Gamma=3$, and 4 out of 10 instances for $\Gamma=6$, while the HKW solves none of them, ending with average gaps of 64\% and 66\%, respectively. Larger instances cannot be solved by EA and HKW ends up with gaps above 90\%.

\section{Conclusions}
\label{sec:conclusions}

Min-max-min combinatorial optimization problems form a class of notoriously difficult optimization problems. In this manuscript, we have provided a first step towards their efficient solution, focusing on the case of budgeted uncertainty. In this work we derived three fast algorithms to solve these problems exactly and heuristically, two of them based on a mixed-integer non-linear reformulation while the third is a discrete enumeration scheme involving ad-hoc dominance rules. The proposed exact algorithms have overcome the difficulties encountered by the previous literature \citep{hanasusanto2015k,subramanyam2017k} for $k=2$, by solving problems in a couple of minutes (often seconds) that were unsolvable in one or two hours using the previous approaches. We also found encouraging results for $k=3$, solving many unsolved instances to optimality in a short amount of time. While our exact approaches can hardly handle larger values of $k$, our local search heuristic based on the non-linear reformulation performs well, providing near-optimal solutions quickly.

In the future, we intend to develop solution algorithms able to solve the problem exactly for larger values of $k$ and more general uncertainty sets. One idea is to use strong integer programming tools based on extended formulations. Preliminary results seem to indicate that the linear programming relaxation of these formulations are strong, hopefully leading to efficient branch-and-bound algorithms.

\section{Acknowledgements}

This research benefited from the support of the FMJH Program PGMO and from the support to this program from EDF-THALES-ORANGE- CRITEO. The authors are also grateful to the referees for their helpful suggestions.

\end{document}